\documentclass[11pt]{amsart}

\usepackage[colorlinks=true, pdfstartview=FitV, linkcolor=blue, citecolor=blue, urlcolor=blue]{hyperref}

\usepackage{amsfonts}
\usepackage{amscd}
\usepackage{amssymb}
\usepackage{amsthm}
\usepackage{amsmath}
\usepackage{comment}
\usepackage{epsfig}
\usepackage[all]{xy}
\usepackage[pdftex]{color} 

\usepackage{tikz}

\newcommand{\uid}{
\begin{tikzpicture}[scale=0.5]
\fill[black!07!white] (0,-0.01) .. controls (0,1) and (0,1).. (0,2.01) ..  controls (0.5,2.01) and (0.5,2.01) .. (1,2.01) .. controls (1,1) and (1,1) .. (1,-0.01);
 ..  controls (0.5,0.8) and (0.5,0.8) .. (0,0.8) ;
\draw[line width = 0.06cm] (0,-0.01) -- (0,2.02);
\draw[line width = 0.06cm] (1,-0.01) -- (1,2.02);
\end{tikzpicture}}

\newcommand{\psuv[1]}{
\begin{tikzpicture}[scale=0.7]
\fill[black!35!white] (0,0.3) .. controls (0.5,0.3) and (0.5,0.3).. (1,0.3) ..  controls (1,0.55) and (1,0.55) .. (1,0.8)
 ..  controls (0.5,0.8) and (0.5,0.8) .. (0,0.8) ;
\draw[line width = 0.06cm] (0,0.3) -- (0,0.8);
\draw[line width = 0.06cm] (1,0.3) -- (1,0.8);
\put(0.25,0.33){$#1$};
\draw[line width = 0.06cm] (0,0.55) -- (0.25,0.55);
\draw[line width = 0.06cm] (0.75,0.55) -- (1,0.55);
\draw[->, line width = 0.06cm] (0.5,0.8) -- (0.5,1.1);
\draw[->, line width = 0.06cm, white] (0.5,0.3) -- (0.5,0);
\end{tikzpicture}
}

\newcommand{\psdv[1]}{
\begin{tikzpicture}[scale=0.7]
\fill[black!35!white] (0,0.3) .. controls (0.5,0.3) and (0.5,0.3).. (1,0.3) ..  controls (1,0.55) and (1,0.55) .. (1,0.8)
 ..  controls (0.5,0.8) and (0.5,0.8) .. (0,0.8) ;
\draw[line width = 0.06cm] (0,0.3) -- (0,0.8);
\draw[line width = 0.06cm] (1,0.3) -- (1,0.8);
\put(0.25,0.33){$#1$};
\draw[line width = 0.06cm] (0,0.55) -- (0.25,0.55);
\draw[line width = 0.06cm] (0.75,0.55) -- (1,0.55);
\draw[->, line width = 0.06cm, white] (0.5,0.8) -- (0.5,1.1);
\draw[->, line width = 0.06cm] (0.5,0.3) -- (0.5,0);
\end{tikzpicture}
}

\newcommand{\Dtst}{
 \begin{tikzpicture}[scale=0.5]
\fill[black!35!white] (0,3.01) .. controls (0,2) and (0.5,2).. (1.5,1.5) ..  controls (1,2) and (1,2.5) .. (1,3.01);
\fill[black!35!white] (3,3.01) .. controls (3,2) and (2.5,2) .. (1.5,1.5) .. controls (2,2) and (2,2.5) .. (2,3.01);
\fill[black!07!white] (0,-0.01) .. controls (0,1) and (0.5,1).. (1.5,1.5) ..  controls (1,1) and (1,0.5) .. (1,-0.01);
\fill[black!07!white] (3,-0.01) .. controls (3,1) and (2.5,1) .. (1.5,1.5) .. controls (2,1) and (2,0.5) .. (2,-0.01);
\draw[line width = 0.07cm, white] (1.1,1.69) -- (1.9, 1.45);
\draw[line width = 0.06cm] (0,3) .. controls (0,2) and (0.5,2) .. (1.5,1.5) .. controls (2.5,1) and (3,1) .. (3,0);
\draw[line width = 0.08cm, white] (1.1,1.69) -- (1.9, 1.31);
\draw[line width = 0.06cm] (1,3) .. controls (1,2.5) and (1,2) .. (1.5,1.5) .. controls (2,1) and (2,0.5) .. (2,0);
\draw[line width = 0.08cm, white] (1.35,1.66) -- (1.65, 1.33);
\draw[line width = 0.06cm] (2,3) .. controls (2,2.5) and (2,2) .. (1.5,1.5) .. controls (1,1) and (1,0.5) .. (1,0);
\draw[line width = 0.07cm, white] (1.625,1.654) -- (1.383, 1.355);
\draw[line width = 0.06cm] (3,3) .. controls (3,2) and (2.5,2) .. (1.5,1.5) .. controls (0.5,1) and (0,1) .. (0,0);

  \end{tikzpicture}
}

\newcommand{\twist}{
\begin{picture}(1.5,1.4)
\put(0,-0.8){
 \begin{tikzpicture}[scale=0.5]
\fill[black!35!white] (0,-0.01) .. controls (0,0.5) and (0,0.5) .. (0.5,1) .. controls (1,0.5) and (1,0.5) .. (1,-0.01);
\fill[black!07!white] (0,2.01) .. controls (0,1.5) and (0,1.5) .. (0.5,1) .. controls (1,1.5) and (1,1.5) .. (1,2.01);
\draw[line width = 0.06cm] (0,-0.01) .. controls (0,0.5) and (0,0.5) .. (0.5,1) .. controls (1,1.5) and (1,1.5) .. (1,2);
\draw[line width = 0.07cm, white] (0.4,0.9) .. controls (0.5,1) and (0.5,1) .. (0.6,1.1);
\draw[line width=0.06cm] (1,-0.01) .. controls (1,0.5) and (1,0.5) .. (0.5,1) .. controls (0,1.5) and (0,1.5) .. (0,2);
  \end{tikzpicture}}
  \end{picture}}
  
  \newcommand{\surtwist}{
\begin{picture}(1.5,1.4)
\put(0,-0.8){
 \begin{tikzpicture}[scale=0.5, xscale=-1]
\fill[black!35!white] (0,-0.01) .. controls (0,0.5) and (0,0.5) .. (0.5,1) .. controls (1,0.5) and (1,0.5) .. (1,-0.01);
\fill[black!07!white] (0,2.01) .. controls (0,1.5) and (0,1.5) .. (0.5,1) .. controls (1,1.5) and (1,1.5) .. (1,2.01);
\draw[line width = 0.06cm] (0,-0.01) .. controls (0,0.5) and (0,0.5) .. (0.5,1) .. controls (1,1.5) and (1,1.5) .. (1,2);
\draw[line width = 0.07cm, white] (0.4,0.9) .. controls (0.5,1) and (0.5,1) .. (0.6,1.1);
\draw[line width=0.06cm] (1,-0.01) .. controls (1,0.5) and (1,0.5) .. (0.5,1) .. controls (0,1.5) and (0,1.5) .. (0,2);
  \end{tikzpicture}}
  \end{picture}}

    \newcommand{\sultwist}{
\begin{picture}(1.5,1.4)
\put(0,-0.8){
 \begin{tikzpicture}[scale=0.5, xscale=1]
\fill[black!35!white] (0,-0.01) .. controls (0,0.5) and (0,0.5) .. (0.5,1) .. controls (1,0.5) and (1,0.5) .. (1,-0.01);
\fill[black!07!white] (0,2.01) .. controls (0,1.5) and (0,1.5) .. (0.5,1) .. controls (1,1.5) and (1,1.5) .. (1,2.01);
\draw[line width = 0.06cm] (0,-0.01) .. controls (0,0.5) and (0,0.5) .. (0.5,1) .. controls (1,1.5) and (1,1.5) .. (1,2);
\draw[line width = 0.07cm, white] (0.4,0.9) .. controls (0.5,1) and (0.5,1) .. (0.6,1.1);
\draw[line width=0.06cm] (1,-0.01) .. controls (1,0.5) and (1,0.5) .. (0.5,1) .. controls (0,1.5) and (0,1.5) .. (0,2);
  \end{tikzpicture}}
  \end{picture}}
  
\newcommand{\usltwist}{
\begin{picture}(1.5,1.4)
\put(0,-0.8){
 \begin{tikzpicture}[rotate=180, scale=0.5]
\fill[black!35!white] (0,-0.01) .. controls (0,0.5) and (0,0.5) .. (0.5,1) .. controls (1,0.5) and (1,0.5) .. (1,-0.01);
\fill[black!07!white] (0,2.01) .. controls (0,1.5) and (0,1.5) .. (0.5,1) .. controls (1,1.5) and (1,1.5) .. (1,2.01);
\draw[line width = 0.06cm] (0,-0.01) .. controls (0,0.5) and (0,0.5) .. (0.5,1) .. controls (1,1.5) and (1,1.5) .. (1,2);
\draw[line width = 0.07cm, white] (0.4,0.9) .. controls (0.5,1) and (0.5,1) .. (0.6,1.1);
\draw[line width=0.06cm] (1,-0.01) .. controls (1,0.5) and (1,0.5) .. (0.5,1) .. controls (0,1.5) and (0,1.5) .. (0,2);
  \end{tikzpicture}}
  \end{picture}}
  
  \newcommand{\usrtwist}{
\begin{picture}(1.5,1.4)
\put(0,-0.8){
 \begin{tikzpicture}[rotate=180, scale=0.5, xscale=-1]
\fill[black!35!white] (0,-0.01) .. controls (0,0.5) and (0,0.5) .. (0.5,1) .. controls (1,0.5) and (1,0.5) .. (1,-0.01);
\fill[black!07!white] (0,2.01) .. controls (0,1.5) and (0,1.5) .. (0.5,1) .. controls (1,1.5) and (1,1.5) .. (1,2.01);
\draw[line width = 0.06cm] (0,-0.01) .. controls (0,0.5) and (0,0.5) .. (0.5,1) .. controls (1,1.5) and (1,1.5) .. (1,2);
\draw[line width = 0.07cm, white] (0.4,0.9) .. controls (0.5,1) and (0.5,1) .. (0.6,1.1);
\draw[line width=0.06cm] (1,-0.01) .. controls (1,0.5) and (1,0.5) .. (0.5,1) .. controls (0,1.5) and (0,1.5) .. (0,2);
  \end{tikzpicture}}
  \end{picture}}
  
\newcommand{\ptwist}{
 \begin{tikzpicture}[scale=0.5]
\fill[black!35!white] (0,-0.01) .. controls (0,0.5) and (0,0.5) .. (0.5,1) .. controls (1,0.5) and (1,0.5) .. (1,-0.01);
\fill[black!07!white] (0,2.01) .. controls (0,1.5) and (0,1.5) .. (0.5,1) .. controls (1,1.5) and (1,1.5) .. (1,2.01);
\draw[line width = 0.06cm] (0,-0.01) .. controls (0,0.5) and (0,0.5) .. (0.5,1) .. controls (1,1.5) and (1,1.5) .. (1,2);
\draw[line width = 0.07cm, white] (0.4,0.9) .. controls (0.5,1) and (0.5,1) .. (0.6,1.1);
\draw[line width=0.06cm] (1,-0.01) .. controls (1,0.5) and (1,0.5) .. (0.5,1) .. controls (0,1.5) and (0,1.5) .. (0,2);
  \end{tikzpicture}}

   \newcommand{\halfpusltwist}{
\begin{tikzpicture}[rotate=180, xscale=-0.5,yscale=0.25]
\fill[black!35!white] (0,-0.01) .. controls (0,0.5) and (0,0.5) .. (0.5,1) .. controls (1,0.5) and (1,0.5) .. (1,-0.01);
\fill[black!07!white] (0,2.01) .. controls (0,1.5) and (0,1.5) .. (0.5,1) .. controls (1,1.5) and (1,1.5) .. (1,2.01);
\draw[line width = 0.06cm] (0,-0.01) .. controls (0,0.5) and (0,0.5) .. (0.5,1) .. controls (1,1.5) and (1,1.5) .. (1,2);
\draw[line width = 0.09cm, white] (0.3,0.88) .. controls (0.5,1) and (0.5,1) .. (0.7,1.15);
\draw[line width=0.06cm] (1,-0.01) .. controls (1,0.5) and (1,0.5) .. (0.5,1) .. controls (0,1.5) and (0,1.5) .. (0,2);
  \end{tikzpicture}
 }
 
    \newcommand{\halfpsultwist}{
\begin{tikzpicture}[xscale=-0.5,yscale=0.25]
\fill[black!35!white] (0,-0.01) .. controls (0,0.5) and (0,0.5) .. (0.5,1) .. controls (1,0.5) and (1,0.5) .. (1,-0.01);
\fill[black!07!white] (0,2.01) .. controls (0,1.5) and (0,1.5) .. (0.5,1) .. controls (1,1.5) and (1,1.5) .. (1,2.01);
\draw[line width = 0.06cm] (0,-0.01) .. controls (0,0.5) and (0,0.5) .. (0.5,1) .. controls (1,1.5) and (1,1.5) .. (1,2);
\draw[line width = 0.09cm, white] (0.3,0.88) .. controls (0.5,1) and (0.5,1) .. (0.7,1.15);
\draw[line width=0.06cm] (1,-0.01) .. controls (1,0.5) and (1,0.5) .. (0.5,1) .. controls (0,1.5) and (0,1.5) .. (0,2);
  \end{tikzpicture}
 }

\newcommand{\ucap}{
\begin{tikzpicture}[scale=0.5]
\fill[black!07!white]  (0,0) arc (180:0:1.5cm);
\fill[white]  (1,0) arc (180:0:0.5cm);
\draw[line width = 0.06cm]  (0,0) arc (180:0:1.5cm);
\draw[line width = 0.06cm]  (1,0) arc (180:0:0.5cm);
\end{tikzpicture}}

\newcommand{\halfucap}{
\begin{tikzpicture}[yscale=0.25, xscale=0.5]
\fill[black!07!white]  (0,0) arc (180:0:1.5cm);
\fill[white]  (1,0) arc (180:0:0.5cm);
\draw[line width = 0.06cm]  (0,0) arc (180:0:1.5cm);
\draw[line width = 0.06cm]  (1,0) arc (180:0:0.5cm);
\end{tikzpicture}}

\newcommand{\uudcap[1]}{
\setlength{\unitlength}{0.5cm}
\begin{picture}(3.6,1.4)
\put(2,-1){\pudv[#1]}
\put(0,-0.25){\ucap}
\put(0,-1){\puuv[#1]}
\end{picture}
}

\newcommand{\uducap[1]}{
\setlength{\unitlength}{0.5cm}
\begin{picture}(3.6,1.4)
\put(0,-1){\pudv[#1]}
\put(0,-0.25){\ucap}
\put(2,-1){\puuv[#1]}
\end{picture}
}

\newcommand{\ucup}{
\begin{tikzpicture}[scale=0.5, rotate=180]
\fill[black!07!white]  (0,0) arc (180:0:1.5cm);
\fill[white]  (1,0) arc (180:0:0.5cm);
\draw[line width = 0.06cm]  (0,0) arc (180:0:1.5cm);
\draw[line width = 0.06cm]  (1,0) arc (180:0:0.5cm);
\end{tikzpicture}}

\newcommand{\halfucup}{
\begin{tikzpicture}[xscale=0.5, yscale=0.25, rotate=180]
\fill[black!07!white]  (0,0) arc (180:0:1.5cm);
\fill[white]  (1,0) arc (180:0:0.5cm);
\draw[line width = 0.06cm]  (0,0) arc (180:0:1.5cm);
\draw[line width = 0.06cm]  (1,0) arc (180:0:0.5cm);
\end{tikzpicture}}

\newcommand{\uudcup[1]}{
\setlength{\unitlength}{0.5cm}
\begin{picture}(3.6,1.4)
\put(-0.15,-0.75){
\begin{picture}(3.4,1.4)
\put(0,0.9){\puuv[#1]}
\put(0,-0.25){\ucup}
\put(2,0.9){\pudv[#1]}
\end{picture}}
\end{picture}
}

\newcommand{\uducup[1]}{
\setlength{\unitlength}{0.5cm}
\begin{picture}(3.6,1.4)
\put(-0.15,-0.75){
\begin{picture}(3.4,1.4)

\put(2,0.9){\puuv[#1]}
\put(0,-0.25){\ucup}
\put(0,0.9){\pudv[#1]}
\end{picture}}
\end{picture}
}

\newcommand{\plaincrossing}{
\begin{picture}(3.1,1.3) \put(0,-1.45){
\begin{tikzpicture}[scale=0.5]
\fill[black!07!white] (0,3) .. controls (0,2) and (2,1) .. (2,0) .. controls (2.5,0) and (2.5,0) .. (3,0) .. controls (3,1) and (1,2) .. (1,3);
\draw[line width = 0.06cm] (0,3) .. controls (0,2) and (2,1) .. (2,0);
\draw[line width = 0.06cm] (1,3) .. controls (1,2) and (3,1) .. (3,0);
\draw[line width = 0.6cm, white] (0.5,0) .. controls (0.5,1) and (2.5,2) .. (2.5,3);

\fill[black!07!white] (0,0) .. controls (0,1) and (2,2) .. (2,3) .. controls (2.5,3) and (2.5,3) .. (3,3) .. controls (3,2) and (1,1) ..  (1,0);

\draw[line width = 0.06cm] (0,0) .. controls (0,1) and (2,2) .. (2,3);
\draw[line width = 0.06cm] (1,0) .. controls (1,1) and (3,2) .. (3,3);
\end{tikzpicture}}
\end{picture}}

\newcommand{\pplaincrossing}{
\begin{tikzpicture}[scale=0.5]
\fill[black!07!white] (0,3) .. controls (0,2) and (2,1) .. (2,0) .. controls (2.5,0) and (2.5,0) .. (3,0) .. controls (3,1) and (1,2) .. (1,3);
\draw[line width = 0.06cm] (0,3) .. controls (0,2) and (2,1) .. (2,0);
\draw[line width = 0.06cm] (1,3) .. controls (1,2) and (3,1) .. (3,0);
\draw[line width = 0.6cm, white] (0.5,0) .. controls (0.5,1) and (2.5,2) .. (2.5,3);
\fill[black!07!white] (0,0) .. controls (0,1) and (2,2) .. (2,3) .. controls (2.5,3) and (2.5,3) .. (3,3) .. controls (3,2) and (1,1) ..  (1,0);
\draw[line width = 0.06cm] (0,0) .. controls (0,1) and (2,2) .. (2,3);
\draw[line width = 0.06cm] (1,0) .. controls (1,1) and (3,2) .. (3,3);
\end{tikzpicture}}

\newcommand{\halfpplaincrossing}{
\begin{tikzpicture}[xscale=0.5,yscale=0.25]
\fill[black!07!white] (0,3) .. controls (0,2) and (2,1) .. (2,0) .. controls (2.5,0) and (2.5,0) .. (3,0) .. controls (3,1) and (1,2) .. (1,3);
\draw[line width = 0.06cm] (0,3) .. controls (0,2) and (2,1) .. (2,0);
\draw[line width = 0.06cm] (1,3) .. controls (1,2) and (3,1) .. (3,0);
\draw[line width = 0.45cm, white] (0.5,0) .. controls (0.5,1) and (2.5,2) .. (2.5,3);
\fill[black!07!white] (0,0) .. controls (0,1) and (2,2) .. (2,3) .. controls (2.5,3) and (2.5,3) .. (3,3) .. controls (3,2) and (1,1) ..  (1,0);
\draw[line width = 0.06cm] (0,0) .. controls (0,1) and (2,2) .. (2,3);
\draw[line width = 0.06cm] (1,0) .. controls (1,1) and (3,2) .. (3,3);
\end{tikzpicture}}

\newcommand{\otherhalfpplaincrossing}{
\begin{tikzpicture}[xscale=-0.5,yscale=0.25]
\fill[black!07!white] (0,0) .. controls (0,1) and (2,2) .. (2,3) .. controls (2.5,3) and (2.5,3) .. (3,3) .. controls (3,2) and (1,1) ..  (1,0);
\draw[line width = 0.06cm] (0,0) .. controls (0,1) and (2,2) .. (2,3);
\draw[line width = 0.06cm] (1,0) .. controls (1,1) and (3,2) .. (3,3);
\draw[line width = 0.45cm, white] (2.5,0) .. controls (2.5,1) and (0.5,2) .. (0.5,3);
\fill[black!07!white] (0,3) .. controls (0,2) and (2,1) .. (2,0) .. controls (2.5,0) and (2.5,0) .. (3,0) .. controls (3,1) and (1,2) .. (1,3);
\draw[line width = 0.06cm] (0,3) .. controls (0,2) and (2,1) .. (2,0);
\draw[line width = 0.06cm] (1,3) .. controls (1,2) and (3,1) .. (3,0);

\end{tikzpicture}}

\newtheorem{Theorem}{Theorem}[section]
 
\newtheorem{Lemma}[Theorem]{Lemma}

\newtheorem{Corollary}[Theorem]{Corollary}

\newtheorem{Definition}[Theorem]{Definition}

\theoremstyle{definition}
\newtheorem{Comment}[Theorem]{Comment}

\newcommand{\arxiv}[1]{\href{http://arxiv.org/abs/#1}{\tt arXiv:\nolinkurl{#1}}}

\def\fsl{\mathfrak{sl}}

\newcommand{\bc}{\mathbb{C}}

\newcommand{\End}{\operatorname{End}}

\newcommand{\Flip}{\mbox{Flip}}

\newcommand{\Id}{\operatorname{Id}}

\newcommand{\DRIBBON}{\mathcal{DRIBBON}}
\newcommand{\RIBBON}{\mathcal{RIBBON}}

\newcommand{\DR}{\mathcal{DRIBBON}}

\newcommand{\cF}{\mathcal{F}}

\newcommand{\ev}{\operatorname{ev}}
\newcommand{\coev}{\operatorname{coev}}
\newcommand{\qtr}{\operatorname{qtr}}
\newcommand{\coqtr}{\operatorname{coqtr}}

\newcommand{\flip}{\operatorname{Flip}}
\newcommand{\trace}{\operatorname{trace}}
\newcommand{\br}{\operatorname{br}}

\theoremstyle{definition}
\newtheorem{Example}[Theorem]{Example}

\newcommand{\mathfig}[2]{{\hspace{-3pt}\begin{array}{c}%
  \raisebox{-2.5pt}{\includegraphics[width=#1\textwidth]{#2}}
\end{array}\hspace{-3pt}}}

\begin{document}

\title[A minus sign]{A minus sign that used to annoy me but now I know why it is there  \\ \small{(two constructions of the Jones polynomial)}}

\author{Peter Tingley}
\address{Department of Mathematics and Statistics, Loyola University, Chicago, IL}
\email{ptingley@luc.edu}

\begin{abstract}
We consider two well known constructions of link invariants. One uses skein theory: you resolve each crossing of the link as a linear combination of things that don't cross, until you eventually get a linear combination of links with no crossings, which you turn into a polynomial. The other uses quantum groups: you construct a functor from a topological category to some category of representations in such a way that (directed framed) links get sent to endomorphisms of the trivial representation, which are just rational functions. Certain instances of these two constructions give rise to essentially the same invariants, but when one carefully matches them there is a minus sign that seems out of place. We discuss exactly how the constructions match up in the case of the Jones polynomial, and where the minus sign comes from. On the quantum group side, one is led to use a non-standard ribbon element, which then allows one to consider a larger topological category. 
\end{abstract}

\maketitle

\tableofcontents

\section*{Introduction} 

This expository article begins by briefly explaining two constructions of the Jones polynomial (neither of which is Jones' original construction \cite{Jones:1985}). The first is the skein-theoretic construction using the Kauffman bracket \cite{Kauffman:1987}. The second is as a $U_q(\fsl_2)$ quantum group link invariant. We then discuss how the two constructions are related. 

The Kauffman bracket is an isotopy invariant of framed links, but the functor used in the quantum group construction involves a category where morphisms are tangles of \emph{directed} framed links. 
However, in the case we consider, the final quantum group invariant does not in fact depend on the directing, and, up to an annoying sign, it agrees with the Kauffman bracket. 
In these notes we explain the annoying sign and describe how the skein relations used in the Kauffman bracket arise naturally in the quantum group construction. We also discuss how to modify the quantum group construction by using the non-standard ribbon element from \cite{half_twist}. In this way one obtains a functor from a category whose morphisms are tangles of \emph{undirected} framed links, and the annoying minus sign disappears!

After developing these ideas, we give one more justifications for using the non-standard ribbon element: it allows one to give an algebraic operation corresponding to twisting a ribbon by 180 degrees. This is discussed in more detail in \cite{half_twist}.

The sign issue discussed here has of course been noticed many times before, and much of the content of these notes can be found in, for instance, \cite[Appendix H]{Ohtsuki}. One can describe the sign precisely, so in a sense there is no problem, but one hopes for a cleaner solution, with fewer (or at least better explained) signs. 
Using the non-standard ribbon element is just one way to achieve this. Another approach, which comes up in \cite{KR:1989, MPS, S}, modifies the braiding instead of the ribbon element; this works, but has the disadvantage that, at $q=0$, the braiding does not descend to the usual symmetric structure. Both this and our approach essentially boil down to the following: One must choose a square root of $q$ both in defining the braiding and in defining the ribbon element, and things work a bit better if one makes different choices (i.e. $\pm q^{1/2}$) in the two places. 
Yet another approach, which is discussed in \cite{CMW}, is to modify the topological category by using  ``disoriented tangles." 

\subsection*{Acknowledgements}
These notes are loosely based on a talk I first gave in 2008 at the University of Queensland in Brisbane Australia, and I thank Murray Elder and Ole Warnaar for organizing that visit. I also thank Noah Snyder for many interesting discussions, Stephen Sawin for comments on an early draft, and Scott Morrison for encouraging me to clean up these notes for publication. This work was partially supported by Australia Research Council grant DP0879951 and NSF grants DMS-0902649 and DMS-1265555.

\section{Knots, links, link diagrams, and some variants}  
A link (as one expects) is a collection of finitely many circles smoothly embedded in ${\Bbb R}^3$ with no intersections. These are considered up to isotopy, which means if you can move between two links without ever having the strands cross then they are the same. 

One can represent a link $L$ with a {\it link diagram}. This is a flattening of the link into the plane, where at each crossing one keeps track of which strand is on top. We will always assume that the curves which appear in the diagram are all smooth, that the diagram only has simple crossings (i.e. only 2 strands can cross at a single point), and that the curves are never tangent. Certainly any link can be represented this way, although this representation is not unique. One important fact about knot theory is that, given any two diagrams that represent the same link, one can be transformed to the other using only the local {\it Reidemeister moves}:

\begin{center}
 \begin{tikzpicture}[scale=0.26]
\draw[line width = 0.06cm] (0,0) .. controls (0,2) and (1,4) .. (2,4) .. controls (3,4) and (3,1) .. (1.1,2.9);

\draw[line width = 0.06cm] (0.65, 3.35) .. controls (0.25, 3.75) and (0,5) .. (0,6); 
\node at (5,2) {=};

\draw[line width = 0.06cm] (6.5,0)--(6.5,6);
\node at (7.3,0) {,};
 \end{tikzpicture}
 \hspace{0.1cm}
\begin{tikzpicture}[scale=0.39]

\draw[line width = 0.06cm] (0,0) .. controls (2,2) and (2,2) ..  (0,4);

\draw[line width = 0.25cm, color=white] (2,0) .. controls (0,2) and (0,2) ..  (2,4);

\draw[line width = 0.06cm] (2,0) .. controls (0,2) and (0,2) ..  (2,4);

\node at (3,2) {=};

\draw[line width = 0.06cm] (4,0)--(4,4);

\draw[line width = 0.06cm] (5,0)--(5,4);

\node at (6,0){,};

\end{tikzpicture}
 \hspace{0.1cm}
\begin{tikzpicture}[scale=0.26]

\draw[line width = 0.06cm] (2,0) .. controls (0,3) and (0,3) ..  (2,6);

\draw[line width = 0.23cm, color=white] (0,0)--(4,6);
\draw[line width = 0.23cm, color=white] (4,0)--(0,6);

\draw[line width = 0.06cm] (0,0)--(4,6);

\draw[line width = 0.2cm, color=white] (4,0)--(0,6);

\draw[line width = 0.06cm] (4,0)--(0,6);

\node at (5,3) {=};

\end{tikzpicture}
\begin{tikzpicture}[scale=0.26]

\draw[line width = 0.06cm] (2,0) .. controls (4,3) and (4,3) ..  (2,6);

\draw[line width = 0.23cm, color=white] (0,0)--(4,6);
\draw[line width = 0.23cm, color=white] (4,0)--(0,6);

\draw[line width = 0.06cm] (0,0)--(4,6);

\draw[line width = 0.2cm, color=white] (4,0)--(0,6);

\draw[line width = 0.06cm] (4,0)--(0,6);

\node at (5,0) {,};

\end{tikzpicture}
 \hspace{0.1cm}
\begin{tikzpicture}[scale=0.26]

\draw[line width = 0.23cm, color=white] (0,0)--(4,6);
\draw[line width = 0.23cm, color=white] (4,0)--(0,6);

\draw[line width = 0.06cm] (0,0)--(4,6);

\draw[line width = 0.2cm, color=white] (4,0)--(0,6);

\draw[line width = 0.06cm] (4,0)--(0,6);

\draw[line width = 0.23cm, color=white] (2,0) .. controls (0,3) and (0,3) ..  (2,6);
\draw[line width = 0.06cm] (2,0) .. controls (0,3) and (0,3) ..  (2,6);

\node at (5,3) {=};

\end{tikzpicture}
\begin{tikzpicture}[scale=0.26]

\draw[line width = 0.23cm, color=white] (0,0)--(4,6);
\draw[line width = 0.23cm, color=white] (4,0)--(0,6);

\draw[line width = 0.06cm] (0,0)--(4,6);

\draw[line width = 0.2cm, color=white] (4,0)--(0,6);

\draw[line width = 0.06cm] (4,0)--(0,6);

\draw[line width = 0.23cm, color=white] (2,0) .. controls (4,3) and (4,3) ..  (2,6);
\draw[line width = 0.06cm] (2,0) .. controls (4,3) and (4,3) ..  (2,6);

\node at (5,0) {.};

\end{tikzpicture}

\end{center}

\noindent However, actually doing so can be difficult. Even more difficult is showing that one {\it cannot} transform one diagram to another. That is, showing that two links are in fact different. To do that, one looks for an invariant: A function on link diagrams which doesn't change when you do a Reidemeister move. Then, if the invariant is different for two diagrams, the corresponding links themselves are different (i.e. not related by isotopy).

In fact, we need a few variants of links/link diagrams. Sometimes we must work with {\it directed} links, which means that each strand gets an arrow pointed along it in one of the two possible directions, and sometimes we work with {\it framed} links, which means links tied out of flat ribbons (so, you can tell if the ribbon gets twisted). If we draw a framed diagram without indicating the framing explicitly, we mean that the ribbon is lying flat on the page; this is usually called the ``blackboard framing." For framed link diagrams, we will assume that all twists occur as full 360 degree twists (this in particular disallows links ties out of m\"obius strips), although this restriction will be weakened slightly in \S\ref{sec:ht}.

It remains true that one can move between any link diagrams for isotopic framed and/or directed links using Reidemeister moves, the only subtlety being that the one strand move becomes

\begin{center}
 \begin{tikzpicture}[scale=0.32, yscale=1]

\draw[line width = 0.6cm] (0,0) .. controls (0,2) and (1,4) .. (3,5) .. controls (5,5) and (4,1) .. (1.1,2.9);

\draw[line width = 0.6cm] (0.65, 3.35) .. controls (0.25, 3.75) and (0,5) .. (0,6); 

\draw[black!07!white, line width = 0.5cm] (0,0) .. controls (0,2) and (1,4) .. (3,5).. controls (5,5) and (4,1) .. (1.1,2.9);

\draw[black!07!white, line width = 0.5cm] (3,5) .. controls (5,5) and (4,1) .. (1.1,2.9);

\draw[black!07!white, line width = 0.5cm] (0.65, 3.35) .. controls (0.25, 3.75) and (0,5) .. (0,6); 

\draw[white, line width = 0.75cm] (0,0) .. controls (0,2) and (1,4) .. (2.65,4.85);
\draw[ line width = 0.6cm] (0,0) .. controls (0,2) and (1,4) .. (2.65,4.85);
\draw[black!07!white, line width = 0.5cm] (0,0) .. controls (0,2) and (1,4) .. (2.67,4.87);

 \draw node at (6,3) {=};
\draw node at (7.5,3) {\surtwist};
\draw node at (7.5,0) {\usrtwist};
\end{tikzpicture}
 \end{center}
 where each side represents a single framed strand.

\section{The Kauffman bracket construction}  \label{Kauffman_bracket} 

Up to a change in the variable $q$, the following is the well known construction of the Kauffman bracket \cite{Kauffman:1987}.

\begin{Definition} \label{Kauffman-simplifications}
Let $L$ be a link diagram. Simplify $L$ using the following relations until the result is a polynomial in $q^{1/2}$ and $q^{-{1/2}}$. That polynomial, denoted by $K(L)$, is the Kauffman bracket of the link diagram.

\begin{enumerate}
\setlength{\unitlength}{0.35cm}
\thicklines
\item  \label{ocross_Kauffman} 
\begin{picture}(15,2)

\put(1,0){
\begin{picture}(3,3)
\put(0,2){\line(1,-1){1.1}}
\put(2,0){\line(1,-1){1.1}}
\put(0,-1){\line(1,1){3}}
\end{picture}
}

\put(5,0.3){$=$}

\put(9,-1){\line(0,1){3}}
\put(12,-1){\line(0,1){3}}

\put(7,0.3){$q^{1/2}$}

\put(14,0.3){$+$}

\put(16,0.3){$q^{-1/2}$}

\put(20.5,-1){\oval(3,2)[t]} 
\put(20.5,2){\oval(3,2)[b]} 

\end{picture}
\vspace{0.5cm}

\item \label{circle_Kauffman}
\begin{picture}(10,3)
\put(4,0.4){\circle{2.8}}

\put(7.5,0.3){$=$}

\put(10,0.3){$-q-q^{-1}$}

\end{picture}
\vspace{0.5cm}

\item If two diagrams are disjoint, their Kauffman brackets multiply.
\end{enumerate}
Note that \eqref{ocross_Kauffman} depends on which strand is on top.
\end{Definition}

The Kauffman bracket is {\bf not} a link invariant; a simple check will show that it fails to respect the one strand Reidemeister move. But, as discussed in the previous section, the one strand Reidemeister move does not hold exactly for framed link diagrams. In fact, the problem is fixed by working with framed links and introducing the following
extra relation (here both sides represent a single framed string):
\begin{equation} \label{third-Kauff} 
\setlength{\unitlength}{0.42cm}
\begin{picture}(6,2.5)
\put(0,0.9){\surtwist}
\put(0,-1.1){\usrtwist}
\put(2.5,0.2){$=$}
\put(3.5,0.2){$-q^{3/2} $}
\put(6,0){\uid}
\put(6,-2){\uid \; .}
\end{picture}
\end{equation}

\vspace{1cm}
\noindent Note that the direction of the twist (clockwise or counter clockwise) matters. 

The following can be verified fairly easily by checking how the Kauffman bracket changes under each Reidemeister move. 

\begin{Theorem} (see e.g. \cite[Theorem 1.10]{Ohtsuki})
The Kauffman bracket from Definition \ref{Kauffman-simplifications} is an isotopy invariant of framed links.  \qed
\end{Theorem}

We actually want an invariant of unframed links, but it is useful to first complicate things by considering links which are both framed and directed. 

\begin{Definition} \label{def_wbits} 
Consider a framed, directed link diagram. 
\begin{enumerate}
\item A positive crossing is a crossing of the form
\begin{center}
\begin{tikzpicture}[scale=0.5]
\selectcolormodel{gray}
\definecolor{lg}{gray}{0.94}

\draw[color=lg, line width=0.4cm] (3.25,0.25)--(0.25,3.25);

\draw[line width = 0.04cm]  (3.5,0.5)--(0.5,3.5);
\draw[line width = 0.04cm] (3,0)--(0,3);

\draw[color=white, line width=0.7cm] (0,0)--(3.5,3.5);

\draw[color=lg, line width=0.4cm] (0.25,0.25)--(3.25,3.25);

\draw[line width = 0.04cm]  (0.5,0)--(3.5,3);
\draw[line width = 0.04cm]  (0,0.5)--(3,3.5);

\draw[line width = 0.05cm, ->] (2.5,2.5)--(3.7,3.7);
\draw[line width = 0.05cm, ->] (1,2.5)--(-0.2,3.7);

\end{tikzpicture}
\end{center}
That is, a crossing such that, if you approach the crossing along the upper ribbon in the chosen direction and leave along the lower ribbon, you have made a left turn.
\item A negative crossing is a crossing of the form
\begin{center}
\begin{tikzpicture}[scale=0.5]
\selectcolormodel{gray}
\definecolor{lg}{gray}{0.94}

\draw[color=lg, line width=0.4cm] (0.25,0.25)--(3.25,3.25);

\draw[line width = 0.04cm]  (0.5,0)--(3.5,3);
\draw[line width = 0.04cm]  (0,0.5)--(3,3.5);

\draw[color=white, line width=0.7cm] (3.5,0)--(0,3.5);

\draw[color=lg, line width=0.4cm] (3.25,0.25)--(0.25,3.25);

\draw[line width = 0.04cm]  (3.5,0.5)--(0.5,3.5);
\draw[line width = 0.04cm] (3,0)--(0,3);

\draw[line width = 0.05cm, ->] (2.5,2.5)--(3.7,3.7);
\draw[line width = 0.05cm, ->] (1,2.5)--(-0.2,3.7);

\end{tikzpicture}
\end{center}
That is, a crossing such that, if you approach the crossing along the upper ribbon in the chosen direction, then leave along the lower ribbon, you have made a right turn.

\item A positive full twist is a twist of the form
\begin{center}

\setlength{\unitlength}{0.42cm}

\begin{picture}(2,5)
\put(-0.3,0.9){ \usrtwist }
\put(0.04,3.2){\surtwist}
\put(1.3,4.3){\vector(0,1){0.8}}

\end{picture}
\end{center}

\item A negative full twist is a twist in the opposite direction to a positive full twist.

\item The writhe of a link diagram $L$, denoted by $w(L)$, is the number of positive crossings minus the number of negative crossings plus the number of positive full twists minus the number of negative full twists. 
\end{enumerate}
\end{Definition}

The following are fundamental results in knot theory, but both can be checked directly. 

\begin{Lemma} (see \cite{Kauffman:1987})
The writhe $w(L)$ is an invariant of directed framed links. \qed
\end{Lemma}

\begin{Theorem} (see \cite[Theorem 1.5]{Ohtsuki}) \label{Jones1}
Let $L$ be any link. Then the Jones polynomial,
\begin{equation}
J(L):= (-q^{3/2})^{-w(L)} K(L),
\end{equation}
is independent of the framing. Hence $J(L)$ is an isotopy invariant of directed (but not framed) links. \qed
\end{Theorem}

\begin{Comment} \label{knots-wd}
It is straightforward to see that positive full twists are sent to positive full twists if the direction of the ribbon is reversed, and positive crossings are sent to positive crossings if the directions of both ribbons involved are reversed. It follows that the choice of directing only affects the Jones polynomial for links with at least two components.
\end{Comment}

\section{The quantum group construction}

Here we describe the Jones polynomial as a $U_q(\fsl_2)$ quantum group link invariant. This uses a circle of ideas developed by a number of authors starting in the late 1980s (see \cite{Turaev} and references therein), making use of the famous Drinfel{$'$}d-Jimbo quantum groups \cite{Drin, Jimbo}. We try to give a  feel for how quantum group invariants work in general, but only fully develop the simplest case.

\subsection{The quantum group $U_q(\fsl_2)$ and its representations}

$U_q(\fsl_2)$ is an infinite dimensional algebra related to the Lie-algebra $\fsl_2$ of $2 \times 2 $ matrices with trace zero. It is the algebra over the field of rational functions $\bc(q)$ generated by $E,F, K$ and $K^{-1}$, subject to the relations
\begin{equation}
\begin{aligned}
&K K^{-1}=1, \\
&KEK^{-1} = q^2 E, \\
& KFK^{-1} = q^{-2} F, \\
& EF-FE= \frac{K-K^{-1}}{q-q^{-1}}.
\end{aligned}
\end{equation}
In some places below we must actually work over $\bc[q^{1/2}]$, which is to say we adjoin a chosen square root of $q$ to the field. 
 
$U_q(\fsl_2)$ has a representation $V_n$ of dimension $n+1$ for each integer $n \geq 0$, which we now describe. Introduce the ``quantum integers"
\begin{equation}
[n]:= \frac{q^n-q^{-n}}{q-q^{-1}} = q^{n-1}+q^{n-3} + \cdots + q^{-n+1}.
\end{equation}
The representation $V_n$ has $\bc(q)$-basis $\{ v_n, v_{n-2}, \cdots, v_{-n+2}, v_{-n} \}$, and the actions of $E,F$ and $K$ are given by
\begin{equation}
\begin{aligned}
E(v_{-n+2j})&=
\begin{cases}
[j+1] v_{-n+2j+2} \quad \text{ if } 0 \leq j < n \\
0 \quad \text{ if } j=n,
\end{cases}
\\
F(v_{n-2j})&=
\begin{cases}
[j+1] v_{n-2j-2} \quad \text{ if } 0 \leq j < n \\
0 \quad \text{ if } j=n,
\end{cases}
\\
K(v_k)&= q^k v_k.
\end{aligned}
\end{equation}
This can be expressed by the following diagram:
\begin{equation}
\setlength{\unitlength}{0.6cm}
\begin{picture}(15,3.5)
\put(1,2){\circle*{0.3}}
\put(3,2){\circle*{0.3}}
\put(5,2){\circle*{0.3}}

\put(11,2){\circle*{0.3}}
\put(13,2){\circle*{0.3}}
\put(15,2){\circle*{0.3}}

\put(7.5,2){$\ldots$}

\put(1.2,2.25){\vector(1,0){1.6}}
\put(3.2,2.25){\vector(1,0){1.6}}
\put(5.2,2.25){\vector(1,0){1.6}}

\put(9.2,2.25){\vector(1,0){1.6}}
\put(11.2,2.25){\vector(1,0){1.6}}
\put(13.2,2.25){\vector(1,0){1.6}}

\put(2.8,1.75){\vector(-1,0){1.6}}
\put(4.8,1.75){\vector(-1,0){1.6}}
\put(6.8,1.75){\vector(-1,0){1.6}}

\put(10.8,1.75){\vector(-1,0){1.6}}
\put(12.8,1.75){\vector(-1,0){1.6}}
\put(14.8,1.75){\vector(-1,0){1.6}}

\put(1.8,2.6){\footnotesize $1$}
\put(3.7,2.6){\footnotesize $[2]$}
\put(5.7,2.6){\footnotesize $[3]$}

\put(9.3,2.6){\footnotesize $[n-2]$}
\put(11.3,2.6){\footnotesize $[n-1]$}
\put(13.7,2.6){\footnotesize $[n]$}

\put(1.7,1){\footnotesize $[n]$}
\put(3.3,1){\footnotesize $[n-1]$}
\put(5.3,1){\footnotesize $[n-2]$}

\put(9.8,1){\footnotesize $[3]$}
\put(11.8,1){\footnotesize $[2]$}
\put(13.8,1){\footnotesize $1$}

\put(0.8,-0.3){$q^n$}
\put(2.8,-0.3){$q^{n-2}$}
\put(4.8,-0.3){$q^{n-4}$}
\put(10.8,-0.3){$q^{-n+4}$}
\put(12.8,-0.3){$q^{-n+2}$}
\put(14.8,-0.3){$q^{-n}$}

\put(-2,2.55){$F:$}
\put(-2,0.95){$E:$}
\put(-2,-0.4){$K:$}

\end{picture}
\end{equation}
\vspace{0.4cm}

There is a tensor product on representations of $U_q(\fsl_2)$, where the action on $a \otimes b \in A \otimes B$ is given by
\begin{equation}
\begin{aligned}
E(a \otimes b) & = Ea \otimes Kb+ a \otimes Eb, \\
F(a \otimes b) & = Fa \otimes b+ K^{-1}a \otimes Fb, \\
K(a \otimes b) & = Ka \otimes Kb.
\end{aligned}
\end{equation}
It turns out that $A \otimes B$ is always isomorphic to $B \otimes A$, and furthermore there is a well known natural system of isomorphisms
\begin{equation}
\sigma_{A,B}^{\br}: A \otimes B \rightarrow B \otimes A
\end{equation}
for each pair $A,B$, called the braiding. A definition of the braiding can be found in, for example \cite{CP} (or Theorem \ref{KR_th} below can also be used as the definition).
Here we only ever apply the braiding to the standard 2-dimensional representations of $U_q(\fsl_2)$, so we can use the following:

\begin{Definition} \label{easy-s-def} Let $V$ be the 2 dimensional representation of $U_q(\fsl_2)$. Use the ordered basis $\{ v_1 \otimes v_1, v_{-1} \otimes v_{1}, v_1 \otimes v_{-1}, v_{-1} \otimes v_{-1} \}$ for $V \otimes V$. Then $\sigma_{V,V}^{\br}: V \otimes V \rightarrow V \otimes V$ is given by the matrix
\begin{equation*}
\sigma^{\br}= 
q^{-1/2} \left(
\begin{array}{cccc}
q&0&0&0 \\
0&q-q^{-1}&1&0\\
0&1&0&0 \\
0&0&0&q
\end{array}
\right).
\end{equation*}
\end{Definition}

There is a standard action of $U_q(\fsl_2)$ on the dual vector space to $V_n$. This is defined using the ``antipode" $S$, which is the algebra anti-automorphism defined on generators by:
\begin{equation}
\begin{aligned}
&S(E)=-EK^{-1}, \\
&S(F)=-K F, \\
&S(K)=K^{-1}.
\end{aligned}
\end{equation}
For $\hat{v} \in V_n^*$ and $X \in U_q(\fsl_2)$, let $X \cdot \hat v$ be the element of $V_n^*$ defined by 
\begin{equation}
(X \cdot \hat v) (w) := \hat v (S(X) w)
\end{equation}
for all $w \in V_n$. It is straightforward to check that this is a left action of $U_q(\fsl_2)$ on $V_n^*$.
It turns out that $V_n$ is always isomorphic to $V_n^*$, which will be important later on. 

\begin{Example} \label{an-iso}
Let $v_1,v_{-1}$ be the basis for $V$. For $i=\pm 1$, let $\hat v_i$ be the element of $V^*$ defined by
\begin{equation}
\hat v_i (v_j)= \delta_{i,j}.
\end{equation}
Calculating using the above definition, the action of $U_q(\fsl_2)$ on $V^*$ is given by
\setlength{\unitlength}{1cm}
\begin{equation}
\begin{picture}(6,0.8)
\put(0,-1){
\begin{picture}(6,2)
\thicklines
\put(2,1){$\hat v_{-1}$}
\put(5,1){$\hat v_1$,}

\put(2.6,1.2){\vector(1,0){2.3}}
\put(4.9,0.9){\vector(-1,0){2.3}}

\put(0.5,1.5){$F:$}
\put(3.5,1.5){$-q^{-1}$}
\put(0.5,0.35){$E:$}
\put(3.5,0.35){$-q^{}$}
\end{picture}}
\end{picture}
\end{equation}

\vspace{0.7cm}

\noindent Consider the map of vector spaces $f: V \rightarrow V^*$ defined by
\begin{equation}
\begin{cases}
f(v_1)= \hat v_{-1} \\
f(v_{-1}) = -q^{-1} \hat v_1
\end{cases}
\end{equation}
One can easily check that $f$ is an isomorphism of $U_q(\fsl_2)$ representations.
\end{Example} 

\begin{Comment}
If one sets $q=1$, the representations $V_n$ described above are exactly the irreducible finite dimensional representations of the usual Lie algebra $\fsl_2$, where one identifies 
\begin{gather}
E \leftrightarrow \left( \begin{array}{cc} 0&1 \\ 0&0 \end{array} \right), \quad F \leftrightarrow \left( \begin{array}{cc} 0&0 \\ 1&0 \end{array} \right), \quad \frac{K-K^{-1}}{q-q^{-1}} \leftrightarrow \left( \begin{array}{cc} 1&0 \\ 0&-1 \end{array} \right).
\end{gather}
Of course, one needs to be a bit careful about interpreting the third identification here, since it looks like you divide by $0$. This issue is addressed in \cite[Chapters 9 and 11]{CP}. 
For us, this is sufficient justification for thinking of $U_q(\fsl_2)$ as related to ordinary $\fsl_2$.
\end{Comment}

\begin{Comment}
Notice that $K$ acts as the identity on all $V_n$ at $q=1$. $U_q(\fsl_2)$ actually has some other finite dimensional representations where $K$ does not act as the identity at $q=1$. So we have not described the full category of finite dimensional representation of $U_q(\fsl_2)$, but only the so called ``type {\bf 1}" representations. The other representations rarely appear in the literature.
\end{Comment}

\subsection{Ribbon elements and quantum traces}

Much of the following can be found in, for example, \cite[Chapter 4]{CP} or \cite{Ohtsuki}. The main difference here is that we work with two ribbon elements throughout. Each satisfies the definition of a ribbon element as in \cite{CP}. Consequently we also have two different quantum traces, and two different co-quantum traces. The non-standard ribbon element $Q_t$ is discussed extensively in \cite{half_twist}.

\begin{Definition} \label{ribbons} The ribbon elements $Q_s$ and $Q_t$ are elements in some completion of $U_q(\fsl_2)$ defined by
\begin{itemize}
\item The standard ribbon element $Q_s$ acts on $V_n$ as multiplication by the scalar $q^{-\frac{n^2}{2}-n}$. 

\item The ``non-standard" or ``half-twist" ribbon element $Q_t$ acts on $V_n$ as multiplication by the scalar $(-1)^nq^{-\frac{n^2}{2}-n}$.
\end{itemize}
\end{Definition}

\begin{Definition} \label{group-likes} The ``grouplike elements" associated to $Q_s$ and $Q_t$ are elements in some completion of $U_q(\fsl_2)$ defined by
\begin{itemize}
\item $g_s$ acts on $v_{n-2j} \in V_n$ as multiplication by $q^{n-2j}$. 

\item $g_t$  acts on $v_{n-2j} \in V_n$ as multiplication by $(-1)^n q^{n-2j}$. 
\end{itemize}
\end{Definition}

\begin{Comment}
The grouplike elements in Definition \ref{group-likes} are related to the ribbon elements in Definition \ref{ribbons} as described in \cite[Chapter 4.2C]{CP}. 
\end{Comment}

\begin{Definition} \label{ev-maps} (see \cite[Section 4.2]{Ohtsuki}) 
 Define the following maps:
\begin{enumerate}
\item $\ev$ is the evaluation map $V^* \otimes V \rightarrow \bc(q)$.

\item $\qtr_{Q_s}$ is the standard quantum trace map $V \otimes V^* \rightarrow \bc(q)$ defined by, for $\phi \in \End(V)=V \otimes V^* $,   $\qtr_{Q_s}(\phi)= \trace(\phi \circ g_s)$. 

\item $\qtr_{Q_t}$ is the ``half-twist" quantum trace map $V \otimes V^* \rightarrow \bc(q)$ defined by, for $\phi \in \End(V)=V \otimes V^* $,   $\qtr_{Q_t}(\phi)= \trace(\phi \circ g_t)$.

\item $\coev$ is the coevaluation map $\bc(q) \rightarrow V \otimes V^*$ defined by $\coev(1)= \Id$, where $\Id$ is the identity map in $\End(V)= V \otimes V^*$.

\item $\coqtr_{Q_s}$ is the standard co-quantum trace map $\bc(q) \rightarrow V^* \otimes V$ defined by $\coqtr_{Q_s}(1) = (1 \otimes g_s^{-1})  \circ \Flip \circ \coev (1),$ where $\Flip$ means interchange the two tensor factors. 

\item $\coqtr_{Q_t}$ is the ``half-twist" co-quantum trace map $\bc(q) \rightarrow V^* \otimes V$ defined by $\coqtr_{Q_t}(1) = (1 \otimes g_t^{-1})  \circ \Flip \circ \coev (1).$
\end{enumerate}
\end{Definition}

\begin{Comment} Although this may not be obvious, the maps in Definition \ref{ev-maps} are all morphisms of $U_q(\fsl_2)$ representations. This can be checked directly.
\end{Comment}

\begin{Comment}
It is often useful to express the maps from Definition \ref{ev-maps} explicitly. One finds that, for all $f \in V^*$ and $v \in V$,  
\begin{equation}
\begin{aligned}
& \ev (f \otimes v) = f(v), \\
& \qtr_Q( v \otimes f) = f(gv), \\
& \coev(1)= \sum_i e_i \otimes e^i, \\
& \coqtr_Q(1) = \sum_i e^i \otimes g^{-1} e_i. 
\end{aligned}
\end{equation}
Here $\{ e^i \}$ and  $\{ e_i \}$ are any dual bases for $V^*$ and $V$. 
One can choose $Q$ to be either $Q_s$ or $Q_t$, and then one must use the grouplike element $g_s$ or $g_t$ accordingly. 
\end{Comment}

\subsection{Two topological categories}
Quantum group knot invariants work by constructing a functor from a certain topological category to the category of representations of the quantum group. We now define the relevant topological category. In fact, we need two slightly different topological categories. 

\setlength{\unitlength}{0.6cm}
\begin{Definition}
 $\DRIBBON$ (directed orientable topological ribbons) is the category where:
 
$\bullet$ An object consists of a finite number of disjoint closed intervals on the real line each directed either up or down. These are considered up to isotopy of the real line. For example:
\vspace{-0.7cm}
\begin{center}
$$\uuv[] \uuv[] \udv[] \udv[] \udv[] \uuv[].$$
\end{center}
`
$\bullet$ A morphism between two objects $A$ and  $B$ consists of a ``tangle of orientable, directed ribbons"  in ${\Bbb R}^2 \times I$, whose ``loose ends" are exactly $(A, 0, 0)  \cup (B,0, 1) \subset {\Bbb R} \times {\Bbb R} \times I $, such that the direction (up or down) of each interval in $A \cup B$ agrees with the direction of the ribbon whose end lies at that interval. 
These are considered up to isotopy. For technical details of the definition of  ``a ribbon", see \cite{CP}.

$\bullet$ Composition of two morphisms is given by stacking them on top of each other, and then shrinking the vertical axis by a factor of two. For example,
\setlength{\unitlength}{0.5cm}
\begin{center}
\begin{picture}(15,3.8)

\put(2,-0.7){\pudv[]}
\put(2,2.7){\puuv[]}
\put(-0.04,-0.5){\pplaincrossing}
\put(0,2.7){\pudv[]}
\put(0,-0.7){\puuv[]}

\put(4.5,1.5){$\circ$}

\put(5.2,2.5){$\uudcup[]$}

\put(10, 1.5){$=$}

\put(14,2.7){\puuv[]}
\put(12.012,1.15){\halfpplaincrossing}
\put(12,2.7){\pudv[]}

\put(12.05,0.8){\halfucup.}

\end{picture}
\end{center}

$\bullet$ This is a monoidal category, where the identity object is ``zero intervals" and the tensor product just places objects and morphisms next to each other. 
\end{Definition}

\vspace{0.15cm}

\begin{Definition}
$\RIBBON$ (undirected orientable topological ribbons) is the category obtained from $\DRIBBON$ by forgetting the directings. So an object consists of a finite number of disjoint closed intervals on the real line, a morphism consists of a tangle of undirected ribbons, and composition is still stacking of tangles. 
\end{Definition}

\subsection{The functor}

The following holds in much greater generality than stated here.

\begin{Theorem} \label{ribbonfunctor} (see \cite[Theorem 5.3.2]{CP}) Let $V$ be the standard 2 dimensional representation of $U_q(\fsl_2)$. For each ribbon element $Q$ (i.e. $Q_s$ or $Q_t$), there is a unique monoidal functor $\cF_Q$ from $\DR$ to $U_q(\fsl_2) \text{-rep}$ such that
\begin{enumerate}

\item $\cF_Q(\uuv[])=V$ and $\cF_Q(\udv[])=V^*$,

\item $\begin{array}{ll} \cF_Q \left( \uducap[]  \right)= \ev, & \cF_Q \left( \uudcap[] \right)= \qtr_Q, \\ \\
\cF_Q \left( \uudcup[] \right)= \coev, & \cF_Q \left( \uducup[] \right)= \coqtr_Q,
\end{array}$
\vspace{0.3cm}

\setlength{\unitlength}{0.5cm}
\item \label{ft} $\cF_Q \left( \begin{picture}(2.2,2.2)
\put(0,0.9){\twist}
\put(0,-1.1){\usltwist}
\end{picture} \right) = Q$ as an automorphism of either $V$ or $V^*$. 

\vspace{0.15in}

\setlength{\unitlength}{0.7cm}
\item \label{eq:asd} $\cF_Q \left( \plaincrossing \right) = \sigma^{\br}$

\noindent as a morphism from the tensor product of the bottom two objects to the tensor product of the top two objects, regardless of the directions of the ribbons. 

\end{enumerate}
The object consisting of no intervals is sent to the ``trivial" 1-dimensional representation $V_0$. 
\qed
\end{Theorem}

\begin{Comment}
Since we only explicitly defined $\sigma^{\br}$ acting on $V \otimes V$, one must be cautious in interpreting \eqref{eq:asd}
when one or both of the ribbons is directed down: One must first choose an explicit isomorphism from $V^*$ to $V$, apply $\sigma^{\br}$, then apply the inverse of that morphism. By naturality, the resulting morphism $\sigma^{\br}$ will not depend on this choice. See Example \ref{qtr-Ex}.
\end{Comment}

\begin{Comment}
Theorem \ref{ribbonfunctor} can in theory be proven by directly verifying invariance under various local isotopies of the ``tangle" diagram, but in fact the usual method is much cleverer, and uses the fact that our morphisms (braiding, ribbon element, evaluation and so on) are defined on all representations, not just $V_1$. In particular, it is useful to consider $\sigma_{V \otimes V, V}^{\br}$. 
\end{Comment}

For any directed framed link $L$, one can draw $L$ as a composition of the elementary features in Theorem \ref{ribbonfunctor}, and hence find the morphism associated to $L$. This is a morphism from the identity object to itself in the category of $U_q(\fsl_2)$ representations, which is just multiplication by a rational function in $q^{1/2}$ (which turns out to be a Laurent polynomial in $q^{1/2}$). By Theorem \ref{ribbonfunctor}, $\cF_Q$ is well defined up to isotopy, so $\cF_Q(L)$, is an isotopy invariant. 

\begin{Example} \label{qtr-Ex}
Here is a way to verify the definition of quantum trace. Recall that $\cF_Q$ is supposed to be defined on $\DRIBBON$, and morphisms there are ribbon tangles \emph{up to isotopy}. One can use an isotopy to change a right going cap to the composition of a twist, a crossing, and a left going cap. But we have only defined $\sigma^{\br}$ explicitly on $V$, not $V^*$, so we insert copies of the isomorphisms $f: V \rightarrow V^*$ and $f^{-1}: V^* \rightarrow V$ (see Example \ref{an-iso}). By naturality of $\sigma^{\br}$, this should not change anything. Diagrammatically,

\setlength{\unitlength}{0.5cm}
\thicklines
\begin{center}
\begin{picture}(10,9.1)
\selectcolormodel{gray}
\definecolor{lg}{gray}{0.93}

\put(7.75,2.1){\udv[],}

\put(8.35,2.5){\color{lg} \line(0,1){0.63}}
\put(8.4,2.5){\color{lg} \line(0,1){0.63}}
\put(8.45,2.5){\color{lg} \line(0,1){0.63}}
\put(8.5,2.5){\color{lg} \line(0,1){0.63}}
\put(8.55,2.5){\color{lg} \line(0,1){0.63}}
\put(8.6,2.5){\color{lg} \line(0,1){0.63}}
\put(8.65,2.5){\color{lg} \line(0,1){0.63}}
\put(8.7,2.5){\color{lg} \line(0,1){0.63}}
\put(8.75,2.5){\color{lg} \line(0,1){0.63}}
\put(8.8,2.5){\color{lg} \line(0,1){0.63}}
\put(8.85,2.5){\color{lg} \line(0,1){0.63}}
\put(8.9,2.5){\color{lg} \line(0,1){0.63}}
\put(8.95,2.5){\color{lg} \line(0,1){0.63}}
\put(9.00,2.5){\color{lg} \line(0,1){0.63}}
\put(9.05,2.5){\color{lg} \line(0,1){0.63}}
\put(9.10,2.5){\color{lg} \line(0,1){0.63}}
\put(9.15,2.5){\color{lg} \line(0,1){0.63}}
\put(9.20,2.5){\color{lg} \line(0,1){0.63}}
\put(9.25,2.5){\color{lg} \line(0,1){0.63}}

\put(8.35,7.3){\color{lg} \line(0,1){0.77}}
\put(8.4,7.3){\color{lg} \line(0,1){0.77}}
\put(8.45,7.3){\color{lg} \line(0,1){0.77}}
\put(8.5,7.3){\color{lg} \line(0,1){0.77}}
\put(8.55,7.3){\color{lg} \line(0,1){0.77}}
\put(8.6,7.3){\color{lg} \line(0,1){0.77}}
\put(8.65,7.3){\color{lg} \line(0,1){0.77}}
\put(8.7,7.3){\color{lg} \line(0,1){0.77}}
\put(8.75,7.3){\color{lg} \line(0,1){0.77}}
\put(8.8,7.3){\color{lg} \line(0,1){0.77}}
\put(8.85,7.3){\color{lg} \line(0,1){0.77}}
\put(8.9,7.3){\color{lg} \line(0,1){0.77}}
\put(8.95,7.3){\color{lg} \line(0,1){0.77}}
\put(9.00,7.3){\color{lg} \line(0,1){0.77}}
\put(9.05,7.3){\color{lg} \line(0,1){0.77}}
\put(9.10,7.3){\color{lg} \line(0,1){0.77}}
\put(9.15,7.3){\color{lg} \line(0,1){0.77}}
\put(9.20,7.3){\color{lg} \line(0,1){0.77}}
\put(9.25,7.3){\color{lg} \line(0,1){0.77}}

\put(6,8){\halfucap}

\put(6,5.35){\otherhalfpplaincrossing}

\put(6.05,8){\line(1,0){1.5}}
\put(6.05,7.3){\line(1,0){1.5}}
\put(7.55,7.3){\line(0,1){0.7}}
\put(6.05,8){\line(0,-1){0.7}}

\put(6.55,7.5){\tiny$f$}

\put(8.05,3.8){\line(1,0){1.5}}
\put(8.05,3.1){\line(1,0){1.5}}
\put(9.55,3.1){\line(0,1){0.7}}
\put(8.05,3.8){\line(0,-1){0.7}}

\put(8.4,3.3){\tiny$f^{-1}$}

\put(6.35,2.3){\color{lg} \line(0,1){3.5}}
\put(6.4,2.3){\color{lg} \line(0,1){3.5}}
\put(6.45,2.3){\color{lg} \line(0,1){3.5}}
\put(6.5,2.3){\color{lg} \line(0,1){3.5}}
\put(6.55,2.3){\color{lg} \line(0,1){3.5}}
\put(6.6,2.3){\color{lg} \line(0,1){3.5}}
\put(6.65,2.3){\color{lg} \line(0,1){3.5}}
\put(6.7,2.3){\color{lg} \line(0,1){3.5}}
\put(6.75,2.3){\color{lg} \line(0,1){3.5}}
\put(6.8,2.3){\color{lg} \line(0,1){3.5}}
\put(6.85,2.3){\color{lg} \line(0,1){3.5}}
\put(6.9,2.3){\color{lg} \line(0,1){3.5}}
\put(6.95,2.3){\color{lg} \line(0,1){3.5}}
\put(7,2.3){\color{lg} \line(0,1){3.5}}
\put(7.05,2.3){\color{lg} \line(0,1){3.5}}
\put(7.1,2.3){\color{lg} \line(0,1){3.5}}
\put(7.15,2.3){\color{lg} \line(0,1){3.5}}
\put(7.2,2.3){\color{lg} \line(0,1){3.5}}

\put(6.3,5.8){\line(0,-1){3.5}}
\put(6.33,5.8){\line(0,-1){3.5}}
\put(6.35,5.8){\line(0,-1){3.5}}

\put(7.3,5.8){\line(0,-1){3.5}}
\put(7.33,5.8){\line(0,-1){3.5}}
\put(7.35,5.8){\line(0,-1){3.5}}

\put(8.3,3.1){\line(0,-1){0.7}}
\put(8.33,3.1){\line(0,-1){0.7}}
\put(8.35,3.1){\line(0,-1){0.7}}

\put(9.3,3.1){\line(0,-1){0.7}}
\put(9.33,3.1){\line(0,-1){0.7}}
\put(9.35,3.1){\line(0,-1){0.7}}

\put(8.3,8.1){\line(0,-1){0.8}}
\put(8.33,8.1){\line(0,-1){0.8}}
\put(8.35,8.1){\line(0,-1){0.8}}

\put(9.3,8.1){\line(0,-1){0.8}}
\put(9.33,8.1){\line(0,-1){0.8}}
\put(9.35,8.1){\line(0,-1){0.8}}

\put(8.0,4.75){\halfpsultwist}
\put(8.0,3.78){\halfpusltwist}

\put(5.75,2.1){\uuv[]}

\put(4.6,5){$\simeq$}

\put(0,7.2){\ucap}

\put(1.75,2.1){\udv[]}

\put(0.43,2.3){\color{lg} \line(0,1){4.97}}
\put(0.49,2.3){\color{lg} \line(0,1){5}}
\put(0.55,2.3){\color{lg} \line(0,1){5}}
\put(0.61,2.3){\color{lg} \line(0,1){5}}
\put(0.67,2.3){\color{lg} \line(0,1){5}}
\put(0.73,2.3){\color{lg} \line(0,1){5}}
\put(0.79,2.3){\color{lg} \line(0,1){5}}
\put(0.85,2.3){\color{lg} \line(0,1){5}}
\put(0.91,2.3){\color{lg} \line(0,1){5}}
\put(0.97,2.3){\color{lg} \line(0,1){5}}
\put(1.03,2.3){\color{lg} \line(0,1){5}}
\put(1.09,2.3){\color{lg} \line(0,1){5}}
\put(1.15,2.3){\color{lg} \line(0,1){5}}
\put(1.21,2.3){\color{lg} \line(0,1){5}}
\put(1.27,2.3){\color{lg} \line(0,1){5}}

\put(0.3,2.3){\line(0,1){5}}
\put(0.33,2.3){\line(0,1){5}}
\put(0.35,2.3){\line(0,1){5}}

\put(1.3,2.3){\line(0,1){5}}
\put(1.33,2.3){\line(0,1){5}}
\put(1.35,2.3){\line(0,1){5}}

\put(2.43,2.549){\color{lg} \line(0,1){5.001}}
\put(2.49,2.549){\color{lg} \line(0,1){5.001}}
\put(2.55,2.549){\color{lg} \line(0,1){5.001}}
\put(2.61,2.549){\color{lg} \line(0,1){5.001}}
\put(2.67,2.549){\color{lg} \line(0,1){5.001}}
\put(2.73,2.549){\color{lg} \line(0,1){5.001}}
\put(2.79,2.549){\color{lg} \line(0,1){5.001}}
\put(2.85,2.549){\color{lg} \line(0,1){5.001}}
\put(2.91,2.549){\color{lg} \line(0,1){5.001}}
\put(2.97,2.549){\color{lg} \line(0,1){5.001}}
\put(3.03,2.549){\color{lg} \line(0,1){5.001}}
\put(3.09,2.549){\color{lg} \line(0,1){5.001}}
\put(3.15,2.549){\color{lg} \line(0,1){5.001}}
\put(3.21,2.549){\color{lg} \line(0,1){5.001}}
\put(3.27,2.549){\color{lg} \line(0,1){4.75}}

\put(2.3,2.3){\line(0,1){5}}
\put(2.33,2.3){\line(0,1){5}}
\put(2.35,2.3){\line(0,1){5}}

\put(3.3,2.3){\line(0,1){5}}
\put(3.33,2.3){\line(0,1){5}}
\put(3.35,2.3){\line(0,1){5}}

\put(-0.25,2.1){\uuv[]}

\end{picture}
\end{center}

\vspace{-1cm}

\noindent where the boxes in the diagram mean ``put in a copy of the isomorphism $f$. Such ``tangles with coupons" are defined precisely in e.g. \cite{CP}. Algebraically, this says
 \begin{equation}
\qtr_Q = \ev  \circ \sigma^{\br} \circ  (\Id \otimes Q^{-1}) = \ev \circ (f \otimes \Id) \circ \sigma^{\br} \circ  (\Id \otimes Q^{-1}) \circ (\Id \otimes f^{-1}).
\end{equation}
Since the action of each element on the right side has been explicitly defined, one can now check that the two sides agree on all basis vectors, using either $Q_s$ or $Q_t$.  
\end{Example}

\section{Matching the two constructions} \label{matching}
The ideas in this section can mostly be found in \cite[Appendix H]{Ohtsuki}.

\subsection{The standard relationship (with the minus sign)}

\begin{Theorem} (see \cite[Theorem 4.19]{Ohtsuki}) \label{with_m} Fix a framed link $L$. Then $\cF_{Q_s}(L)$ is independent of the choice of directing of $L$. Furthermore, $\cF_{Q_s}(L)= (-1)^{w(L)+\# L} K(L),$ where $w(L)$ is the writhe of $L$ and $\# L$ is the number of components of $L$. \qed \end{Theorem}

The $(-1)^{w(L)+\# L}$ in Theorem \ref{with_m} is the sign referred to in the title of these notes. It is certainly explicitly defined, so it is not really a problem; just an annoyance. We now show that, by using $Q_t$ in place of $Q_s$, we can get rid of this sign (although in some sense this just moves the annoyance into the definition of the ribbon element).

There are two other good reasons to consider this modification. First, it allows us to construct a functor from a topological category related to framed but undirected links to $U_q(\fsl_2) \text{-rep}$. Second, it allows us to see how the skein relations used in defining the Kauffman bracket arise naturally in the quantum group construction. 

\subsection{The functor from $\RIBBON$}

There is only one ``elementary" object in $\RIBBON$ (the single interval), as opposed to two in $\DRIBBON$ (the single interval, but with two possible directions). Our morphism will send this single interval to the two dimensional representation $V$. We must then send each feature in the knot diagram to a morphism between the appropriate tensor powers of $V$. For instance,
\begin{center}
\ucap
\end{center}
\vspace{0.1cm}
should be sent to a morphism from $V \otimes V$ to the trivial object. This is as opposed to the directed case, where such ``caps" are sent to morphisms from $V^* \otimes V$ or $V \otimes V^*$ to the trivial object. To do this, we will use the fact that, in this case, $V$ is isomorphic to $V^*$  (for instance, via the isomorphism from Example \ref{an-iso}). We obtain:

\begin{Theorem}  \label{Fundir}
\thicklines
Choose an isomorphism $f: V \rightarrow V^*$. There is a unique monoidal functor $\cF_{f}: \RIBBON \rightarrow U_q(\fsl_2) \text{-rep}$ such that
\begin{enumerate}  \setlength{\unitlength}{0.28cm}

\item $\cF_{f}$ takes the object consisting of a single interval to $V$,

\vspace{0.1cm}

\item \label{undir-cap} 
$\displaystyle \cF_{f} \left( 
\begin{picture}(6,2.2)
\put(0,-1){\ucap}
\end{picture}
\right)
=
\ev \circ (f \otimes \Id) = \qtr_{Q_t} \circ (\Id \otimes f): V \otimes V \rightarrow \bc(q)$,

\vspace{0.1cm}

\item \label{undir-cup} $\displaystyle \cF_{f} \left( 
\begin{picture}(6,2.2)
\put(0,-1.2){\ucup}
\end{picture}
\right) 
=
(\Id \otimes f^{-1})\circ  coev  = (f^{-1} \otimes \Id) \circ  coqtr_{Q_t}: \bc(q) \rightarrow V \otimes V, $

\vspace{0.2cm}

 \setlength{\unitlength}{0.42cm} 
\item \label{undir-cross}  $\displaystyle \cF_{f} \left( 
\begin{picture}(5.5,2)
\put(-0.1,-1.1){\plaincrossing}
\end{picture}
\right) =
\sigma^{\br},$

\vspace{0.4cm}

\item \label{undir-twist} \setlength{\unitlength}{0.42cm} $\cF_{f} \left( \begin{picture}(2.4,2.8)
\put(0,0.9){\twist}
\put(0,-1.1){\usltwist}
\end{picture} \right) = Q_t \quad$ (or, equivalently, multiplication by $-q^{-3/2}$).
 \setlength{\unitlength}{0.5cm}
\end{enumerate}
Furthermore, for any link $L$, any choice of directing of $L$, and any choice of $f$, $\cF_{f}(L)= \cF_{Q_t}(L).$ 
\end{Theorem}

\begin{Comment}
If one tries to use $Q_s$ instead of $Q_t$ in Theorem \ref{Fundir}, then the two expressions on the right side of parts \eqref{undir-cap} and \eqref{undir-cup} are off by a minus sign, and the construction does not work. That the two sides of \eqref{undir-cap} and \eqref{undir-cup} agree follows from the fact that $U_q(\fsl_2)\text{-rep}$, along with ``pivotal structure" related to the ribbon element $Q_t$, is unimodal, as defined in \cite{Turaev}. For an explanation of this pivotal structure and a proof that it is unimodal see \cite[Section 5B]{half_twist}. It is also not hard to directly verify that the expressions agree. 
\end{Comment}

\begin{Comment}
Theorem \ref{Fundir} implies that, for any link $L$, $\cF_{f}(L)$ is independent of the chosen isomorphism $f$. However, the functor $\cF_f$ does depend on this choice. For instance, $\cF_f$ applied to a cap clearly depends on $f$.
\end{Comment}

\begin{proof}[Sketch of proof of Theorem \ref{Fundir}]

First verify by a direct calculation that the two expressions on the right for parts \eqref{undir-cap} and \eqref{undir-cup} agree, so $\cF_{f}$ is well defined on framed link diagrams. 

Fix a diagram $L$ and choose a directing of $L$. Insert $ f \circ f^{-1}$ into $\cF_{Q_t}(L)$ somewhere along every segment of $L$ that is directed down. This clearly doesn't change the morphism. By the naturality of $\sigma^{\br}$, 
\begin{equation}
(1 \otimes f) \circ  \sigma^{\br} = \sigma^{\br} \circ (f \otimes 1).
\end{equation}
Also, $f \circ Q_t = Q_t \circ f$. 
Use these relations to pull all the $f$ and $f^{-1}$ through crossings until they are right next to cups and caps or ends. But now you are essentially calculating $\cF_{f}(L)$. Precisely, $\cF_{f} = \cF_{Q_t}$, composed with a copy of $f$ or $f^{-1}$ for every down-directed ending in the chosen directing. Since $\cF_{Q_t}$ is a functor, it follows that $\cF_f$ is as well. 
\end{proof}

\subsection{Appearance of skein relations in $U_q(\fsl_2) \text{-rep}$} \label{ss:qskein}
A simple calculation shows that
\begin{equation} \label{123_circle}
\setlength{\unitlength}{0.35cm}
\thicklines
\cF_{f} \left(
\begin{picture}(3,1.9)
\put(1.5,0.2){\circle{2.8}}
\end{picture} \right)
=
\text{ multiplication by } -q-q^{-1}.
\end{equation}
\noindent Another direct calculation shows that
\begin{equation} \sigma^{\br}= q^{1/2} \Id + q^{-1/2}  (\Id \otimes f^{-1})\circ  \coev  \circ \qtr_{Q_t} \circ (\Id \otimes f): V \otimes V \rightarrow V \otimes V.
\end{equation}
Equivalently, 
\thicklines
\begin{equation} \label{kauffman_g}
\setlength{\unitlength}{0.35cm}
\cF_{f} \left( 
\begin{picture}(3.5,1.9)
\put(0.2,-1.3){\line(1,1){3}}
\put(0.2,1.7){\line(1,-1){1.1}}
\put(2.2,-0.3){\line(1,-1){1.1}}
\end{picture}
\right) 
=
q^{1/2}
\cF_{f} \left(
\begin{picture}(3.5,1.9)
\put(0.2,-1.3){\line(0,1){3}}
\put(3.2,-1.3){\line(0,1){3}}
\end{picture}
\right)
+
q^{-1/2}
\cF_{f}
\left(
\begin{picture}(3.5,1.9)
\put(1.7,-1.3){\oval(3,2)[t]} 
\put(1.7,1.7){\oval(3,2)[b]} 
\end{picture}
\right).
\end{equation}

\noindent These are exactly the relations defining the Kauffman bracket (Definition \ref{Kauffman-simplifications})! 

\subsection{Fixing the minus sign}
Equation \eqref{third-Kauff} and Theorem \ref{Fundir}\eqref{undir-twist} are identical.
Along with the statements in \S\ref{ss:qskein} 
, this implies that the map from framed link diagrams to polynomials defined by $L \rightarrow \cF_{f}(L)$ satisfies all the skein relations used to calculate the Kauffman bracket $K(L)$, and hence must agree exactly with $K(L)$. That is:
 
\begin{Corollary} \label{get_rid}
Let $L$ be a framed link. Then $\cF_{Q_t}(L)$ is independent of the chosen directing, and is equal to the Kauffman bracket $K(L)$. \qed
\end{Corollary}

\begin{Comment}
Non-standard ribbon elements exist in many cases beyond $U_q(\fsl_2)$, and can also be used to simplify the correspondence between various constructions of link polynomials in those cases. 
\end{Comment}

\section{Another advantage: the half twist} \label{sec:ht}

Consider the following element $X$ in a certain completion of $U_q(\fsl_2)$:

\begin{Definition} \label{defX} $X$ is defined to act on each $V_n$ by $$Xv_{n-2j}  = (-1)^{n-j} q^{\frac{n^2}{4}+\frac{n}{2}} v_{-n+2j}.$$ 
\end{Definition}

One can easily check that $X^{-2}=Q_t$. Furthermore, work of Kirillov-Reshetikhin \cite[Theorem 3]{KR:1990} and Levendorskii-Soibelman \cite[Theorem 1]{LS} shows that $X$ is related to be braiding $\sigma^{\br}$ as follows (see \cite[Comment 7.3]{Rcommutor} for this exact statement):
\begin{Theorem} \label{KR_th}
$\sigma^{\br} = (X^{-1} \otimes X^{-1}) \circ \flip \circ \Delta(X).$ \qed
\end{Theorem}
\noindent  Here $\Delta(X)$ means ``decompose $V \otimes V$ into irreducible components, and apply $X$ to each," and $Flip$ means interchange tensor factors. 
This can be interpreted via the following isotopy:

\setlength{\unitlength}{0.5cm}
\begin{equation}
\begin{picture}(9,2.23)
\put(0,-2){
\begin{picture}(9,5)
\put(0,0.2){\Dtst}
\put(0,3.2){\ptwist}
\put(2,3.2){\ptwist}

\put(4.5,2.5){$\simeq$}
\put(5.95,-0.3){\pplaincrossing}
\put(6,3.2){\uid}
\put(8,3.2){\uid}

\put(9.6,0.3){.}

\end{picture}}
\end{picture}
\end{equation}
\vspace{0.1cm}

\noindent Here $\Flip \circ \Delta(X)$ should be interpreted as a morphism corresponding to twisting both ribbons at once by 180 degree, as on the bottom of the left side. Putting this together, one may hope that $X$ could be interpreted as an isomorphism, and that the functor $\cF_{Q_t}$ from Theorem \ref{Fundir} could be extended in such a way that
\begin{equation}
\setlength{\unitlength}{0.5cm}
\cF_{Q_t} \left(
 \begin{picture}(2,1.25)
\put(-0.1,-0.1){\sultwist}
\end{picture} 
\right)= X^{-1}.
\end{equation}
In fact, such an extended functor has been defined precisely in \cite{half_twist}, resulting in a functor from a larger category where ribbons are allowed to twist by 180 degrees, not just by 360 degrees (although M\"obius bands are still not allowed). Figure \ref{amorphism} shows an example of a morphism in the resulting category. Notice that elementary objects come in both shaded and unshaded versions.

The construction in \cite{half_twist} can only extend $\cF_{Q_t}$, not $\cF_{Q_s}$. One advantage of having such an extended functor is that, since both $\sigma^{\br}$ and $Q_t$ are constructed in term of the ``half-twist" $X$, there is in some sense one less elementary feature.

\begin{figure}
$$\mathfig{0.5}{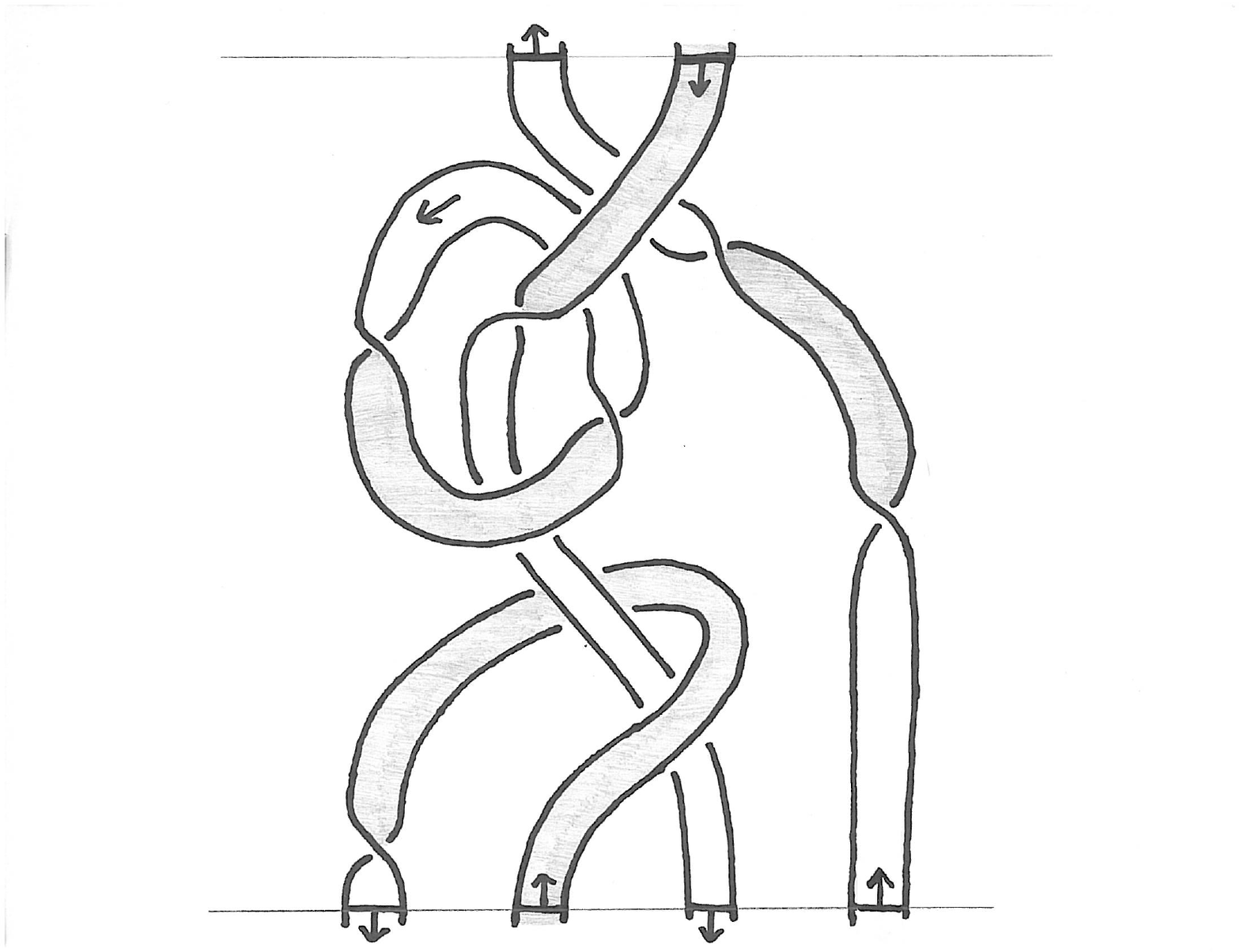}$$

\caption{A morphism in the topological category of ribbons with half twists \label{amorphism}}
\end{figure}

\end{document}